\documentclass{article}

\usepackage{microtype}
\usepackage{graphicx}
\usepackage{subfigure}
\usepackage{booktabs} % for professional tables

\usepackage{hyperref}

\usepackage[accepted]{icml2025}

\usepackage{amsmath}
\usepackage{amssymb}
\usepackage{mathtools}
\usepackage{amsthm}

\usepackage[capitalize,noabbrev]{cleveref}

\theoremstyle{plain}
\newtheorem{theorem}{Theorem}[section]

\theoremstyle{definition}
\newtheorem{definition}[theorem]{Definition}

\theoremstyle{remark}

\usepackage[textsize=tiny]{todonotes}
\usepackage{orcidlink}

\icmltitlerunning{A Hybrid Relaxation-Heuristic Framework for Solving MIP with Binary Variables}

\begin{document}

\twocolumn[
\icmltitle{A Hybrid Relaxation-Heuristic Framework for Solving MIP with Binary Variables}

\begin{icmlauthorlist}
\icmlauthor{Zayn Wang\orcidlink{0009-0006-4133-8181}}{sch}
\end{icmlauthorlist}

\icmlaffiliation{sch}{New York University, New York, New York}

\icmlcorrespondingauthor{Zayn Wang}{zw3858@nyu.edu}

\icmlkeywords{Mixed-Integer Programming (MIP), Linear Relaxation, Lagrange Relaxation, Heuristic Optimization, Portfolio Optimization}

\vskip 0.3in
]



\printAffiliationsAndNotice{}  % leave blank if no need to mention equal contribution

\begin{abstract}
    Mixed-Integer Programming (MIP), particularly Mixed-Integer Linear Programming (MILP) and Mixed-Integer Quadratic Programming (MIQP), has found extensive applications in domains such as portfolio optimization and network flow control, which inclusion of integer variables or cardinality constraints renders these problems NP-hard, posing significant computational challenges. While traditional approaches have explored approximation methods like heuristics and relaxation techniques (e.g. Lagrangian dual relaxation), the integration of these strategies within a unified hybrid framework remains underexplored. In this paper, we propose a generalized hybrid framework to address MIQP problems with binary variables, which consists of two phases: (1) a Mixed Relaxation Phase, which employs Linear Relaxation, Duality Relaxation, and Augmented Relaxation with randomized sampling to generate a diverse pre-solution pool, and (2) a Heuristic Optimization Phase, which refines the pool using Genetic Algorithms and Variable Neighborhood Search (VNS) to approximate binary solutions effectively. Becuase of the page limit, we will only detailedly evaluate the proposed framework on portfolio optimization problems using benchmark datasets from the OR Library, where the experimental results demonstrate state-of-the-art performance, highlighting the framework's ability to solve larger and more complex MIP problems efficiently. This study offers a robust and flexible methodology that bridges relaxation techniques and heuristic optimization, advancing the practical solvability of challenging MIP problems.
\end{abstract}

\section{Introduction}
\label{intro}

\subsection{Mixed-Integer Programming}

The Mixed-Integer Programming (MIP), especially Mixed-Integer Linear Programming (MILP) and Mixed-Integer Quadratic Programming (MIQP), is widely constructed and applied in reality. For example, the most recent advancements include a real-time, mixed-integer programming-based decision-making system for automated driving \citep{quirynen2024realtime}, and optimization in scheduling, such as solving job-shop scheduling problems \citep{ajagekar2022hybrid}. MIQP is also increasingly relevant in machine learning, where it is used for solving decision problems in which the objective function is a machine learning model \citep{anderson2020strong}, processing images and conducting gap analysis \citep{wang2022emotion, wang2024enhancing}, and embedding the model as a layer within neural networks to improve training outcomes \citep{ferber2020mipaal}. However, these are both NP-hard problems caused by the integer variables, or cardinality constraints. The complexity of MIP would increase exponationally due to the property of NP-Hardness. Because of it, people mainly come out with two methods, exact relax methods and heurestic methods, to approach the solution. In the following paper, we will focus on the most widely applied MIP, which is Linear or Quadratic with Binary constraints. And we will regard these problem as MIP for further anlaysis and solving.

Numerous heuristic and meta-heuristic approaches have been applied to MIP optimization. \citet{woodside-oriakhi2011heuristic} explored Genetic Algorithms, Tabu Search, and Simulated Annealing metaheuristics, achieving better results than previous heuristic methods. Building on this, \citet{deng2012markowitzbased} introduced an Improved Particle Swarm Optimization (PSO) technique, which enhanced the robustness and effectiveness, particularly in low-variance conditions. Similarly, \citet{lwin2013hybrid} achieved competitive results using a Hybrid Algorithm. Further advancements came with the Artificial Bee Colony (ABC) Algorithm, where \citet{tuba2014artificial} demonstrated a smaller Euclidean distance between solutions compared to prior methods. This approach was later refined by \citet{kalayci2017artificial}, who added a feasibility enforcement procedure to the ABC Algorithm. Recently, \citet{kalayci2020efficient} reintroduced the Hybrid Algorithm, constructing an efficient metaheuristic combining Ant Colony Optimization, Genetic Algorithms, and ABC Optimization, yielding promising results. In addition to the growth of ABC-based methods, \citet{baykasoglu2015grasp} developed a Greedy Randomized Adaptive Search Procedure (GRASP), while \citet{ertenlice2018survey} explored swarm intelligence (SI) to address time complexity issues. Furthermore, \citet{akbay2020parallel} proposed a new method, the Parallel Variable Neighborhood Search Algorithm, which demonstrated high efficiency in solving MIQP problems.

Although these heuristic methods have gained relatively high results, most only focus on the original primal problem, and their relaxation models also only respect the primal Model. However, the Primal Model itself is not solvable as NP-Hard, and those heuristic methods also cannot guarantee the solutions. Here is where exact methods are introduced. That is, the cardinality constraint could be relaxed into a linear constraint, which transforms the MIQP into a solvable QP. Once the model is solved, the modified linear constraint could be discretized back into the original cardinality constraint. That would be a logical way to solve MIQP models. Considering all the constraints including the cardinality constraint, using Lagrange multipliers, all constraints can be embedded into the objective function, transforming it into a Dual Model. By pre-solving the cardinality constraint or deriving its optimal values ahead to further refine the process, allows MIQP to be solved without the cardinality constraint directly. The Lagrange approach helps derive the expressions for the integer variables. Previous work by \citet{fisher1981lagrangian, nemhauser1988duality} has demonstrated that this Dual Model effectively relaxes the problem, offering a pathway to approach the solution by solving the multiplier variables. Additionally, \citet{li2006optimal, shaw2008lagrangian} have provided substantial evidence supporting the efficiency of this method. Furthermore, recognizing that Lagrange multipliers are more suited for continuous variables, potential inefficiencies might happen when dealing with integer variables. To mitigate errors introduced by the multipliers for discrete variables, the employment of a more advanced relaxation could help the model approach the solution better. \citet{xu2024efficient} proposed a method using the diagonalization of the covariance matrix, establishing lower bounds for the diagonal matrix in correlation and ensuring feasible solutions for MIQP. For MILP, \citep{bragin2024survey} also gives a survey on the recent progress on MILP using Lagrange relaxations and the results of the recent achievement.

In this paper, we propose three types of relaxations to the primal problem: linear relaxation, duality relaxation, and augmented relaxation, resulting in the Line Model, Dual Model, and Augm Model, respectively. These exact relaxations are further enhanced by combining them with heuristic methods to form a new, more effective approach. To ensure comprehensive verification, we define and analyze the gaps between the primal problem and the others models. For the labs and pages limitation, we will only focus on two optimization programmings, portfolio optimization and network flow optimization for experiment and verification.

The remainder of this paper is organized as follows: In the rest subsections of Section \ref{intro}, we would detailedly discuss about how MIPs are used for portfolio optimization and network flow optimization. In Section \ref{formu}, we formally define and formulate our models and gaps for later justification, and in Section \ref{appro} we state the detailed approach to how we get our results using the models. Then in Section \ref{resul} we run our approach and examine the effectiveness of our method. For Section \ref{concl} we drop out the conclusion with our results.

\subsection{Portfolio Optimization Model}

With trillions of dollars circulating in global stock markets, investors are increasingly seeking portfolios that not only generate strong returns but also mitigate risk, making their investments safer. The challenge lies in identifying the optimal combination of assets that maximizes expected returns while minimizing risk, often measured by the variance of the portfolio. Striking the right balance between return and risk is crucial for developing effective portfolio strategies.

\citet{markowitz1952portfolio} introduced a model that examines the relationship between expected return and risk, measured by variance, based on the different weights assigned to securities or assets. This mean-variance (MV) model uses historical prices to estimate expected returns and asset correlations to calculate variance. By optimizing either the sum of expected returns or minimizing total variance, the optimal weight for each asset can be determined. Later, \citet{markowitz1956optimization} extended this work by providing an optimization approach for quadratic functions, specifically variance, subject to linear constraints and a fixed expected return. This formulation, a Quadratic Programming (QP) problem, also introduced the concept of ``efficient points,'' which form what is now known as the Efficient Frontier (EF).

While the Mean-Variance (MV) model is foundational, it lacks realism without incorporating practical constraints. For instance, \citet{graziasperanza1996heuristic} introduced minimum transaction units and bounds on asset weights into the model, and \citet{chang2000heuristics} emphasized the importance of restricting the number of selected assets, making the model more applicable to real-world scenarios. The inclusion of constraints, particularly cardinality constraints, transforms the MV model from a Quadratic Programming (QP) problem into a Mixed-Integer Quadratic Programming (MIQP) problem, which is NP-Hard \citep{pia2017mixedinteger}. Several studies, including \citet{guzelsoy2007duality}, \citet{guzelsoy2011integer}, and \citet{feizollahi2017exact}, have conducted extensive research on MIQP using Lagrange duality models. Building on this work, \citet{shaw2008lagrangian} explored how embedding the constraints into the original primal model through Lagrange multipliers forms a relaxed duality model or Dual Model. Although these relaxations introduce errors and gaps in the solutions, they make it feasible to find solutions within a practical time frame. \citet{shaw2008lagrangian} demonstrated a significant reduction in solution time using Gurobi, and further research by \citet{xu2024efficient} showed that the Dual Model can help establish a lower bound for the primal, facilitating solution approaches. In addition to Dual Models, various heuristic methods have been proposed to tackle more realistic models with constraints.

\begin{definition}
\label{MV}
    The typical MV Model could be formulated as:
    \begin{align}
        \min_{\mathbf{x}, \mathbf{b}} \ & \mathbf{x}^TQ\mathbf{x} \label{mv1} \\
        \mathrm{s.t.} \ & \sum{\mathbf{x}} = \mathbf{1}\mathbf{x} = 1 \label{mv2} \\
        & \mathbf{r}^T\mathbf{x} = r \label{mv3} \\
        & \sum{\mathbf{b}} = \mathbf{1}\mathbf{b} = b \label{mv4} \\
        & l\mathbf{b} \leq \mathbf{x} \leq u\mathbf{b} \label{mv5} \\
        & \mathbf{x} \in \mathbb{R}^n, \mathbf{b} \in \mathbb{B}^n. \label{mv6} 
    \end{align}
\end{definition}where (\ref{mv1}) defines the object, variance, or risk, the programming wants to minimize, and $Q$ is the covariance matrix which is positive-definite; (\ref{mv2}) ensures that the sum of the portion in each portfolio should be $1$; (\ref{mv3}) constraints the expected return of the portfolio; (\ref{mv4}) constraints the number of assets we are able to gain, that is how we will make the choice on assets first; (\ref{mv5}) gives the upper bound and the lower bound, and focus on the assets we only have; (\ref{mv6}) defines the property of assets as reals and choices as binaries.

\section{Problem Formulation}
\label{formu}

\def\x{\mathbf{x}}
\def\q{\mathbf{q}}
\def\R{\mathbb{R}}
\def\B{\mathbb{B}}
\def\c{\mathbf{c}}
\def\b{\mathbf{b}}

\subsection{Primal Model (Prim Model)}
\label{prim}

\begin{theorem}
    For all Primal Model, or original cardinality binary constrained optimization problem, can be defined as:
    \begin{align}
        \text{(Prim)}\min_{\x, \b} \ & \x^TQ\x + q\x \label{prim1} \\
        \mathrm{s.t.} \ & A\x = \c_a \label{prim2} \\
        & L\b \leq \x \leq U\b \label{prim3} \\
        & B\b = \c_b \label{prim4} \\
        & \x \in \R, \b \in \B. \label{prim5}
    \end{align}
    where (\ref{prim1}) reveals the basic LP or QP object to be minimized, and if it is a MILP, then $Q = \mathbf{0}$; (\ref{prim2}) reveals the basic LP or QP constraints; (\ref{prim3}) reveals the upper bound and lower bound for $\x$; (\ref{prim4}) reveals the cardinality constraints; and (\ref{prim5}) reveals the identities for $\x$ and $\b$.
\end{theorem}

\begin{theorem}
    The typical MV Model from (\ref{MV}) could be generalized to the Prim Model with: \begin{align*}
        A & = \left[\begin{array}{l}
            \mathbf{r}^T \\
            \hline
            \mathbf{1}
        \end{array}\right]
        & \mathbf{c}_a & = \left[\begin{array}{l}
            r \\
            \hline
            1
        \end{array}\right] \\
        L & = lI & U & = uI \\
        B & = \mathbf{1} & \mathbf{c}_b & = [b].
    \end{align*}
\end{theorem}

\iffalse
\begin{proof}
    Let \begin{align*}
        \min_{\x \in \R^n \times \B^n} \ & \x^TQ\x + q\x \\
        \mathrm{s.t.} \ & D\x \leq \c_d \\
    \end{align*} be a binary MIP. Then
\end{proof}
\fi

\subsection{Linear Model (Line Model)}
\label{line}

\begin{definition}
    We define the linear model (Line Model in short) as:
    \begin{align}
        \textit{(Line)}\min_{\x, \b_\R} \ & \x^TQ\x + q\x \label{line1} \\
        \mathrm{s.t.} \ & A\x = \c_a \label{line2} \\
        & L\b_\R \leq \x \leq U\b_\R \label{line3} \\
        & B\b_\R = \c_b \label{line4} \\
        & \x, \b_\R \in \R \label{line5}\\
        & \b = \mathbf{1}_{(b_\R)_i \in \mathrm{top}\{\b_\R|B\mathbf{1}_{b_\R} = \c_b\}}. \label{line6}
    \end{align}
    where (\ref{line1}) reveals the object to be minimized; (\ref{line2}) to (\ref{line5}) are similar to (\ref{prim}) but $\b_\R \in \R$; and (\ref{line6}) discretize $\b_\R \in \R$ to $\b \in \B$ by discretize the top multiple $(b_\R)_i$ to $1$ and the rest to $0$ under keeping $B\b = \c_b$.
\end{definition}

\def\l{\mathbf{\lambda}^T}
\subsection{Duality Model (Dual Model)}
\label{dual}

\begin{definition}
    Since the main challanging is about the binary variable $\b$, then we introducing lagrange multiplers to resolve the constraints of $\b$ to define the initial duality model as:
    \begin{align*}
        \max_{\l}\min_{\x, \b} \ & \x^TQ\x + q\x + \l_a(A\x - \c_a) + \l_b(B\b - \c_b) \\ & + \l_l(L\b - \x) + \l_u(\x - U\b)\\
        & = \x^TQ\x + (q + \l_aA - \l_l + \l_u)\x \\ & + (\l_bB + \l_lL - \l_uU) \b - \l_ac_a - \l_bc_b \\
        \mathrm{s.t.} \ & \l_a, \l_b, \l_l, \l_u \in \R, \l_l, \l_u \geq \mathbf{0}.
    \end{align*}
\end{definition}

\begin{theorem}
    If $Q$ is positive-definite, i.e. the MIP is MIQP, the final duality model could be derived as: \begin{align}
            \text{(Dual)}\max_{\l} \ & -\hat{\x}^TQ\hat{\x} + (\l_bB + \l_lL - \l_uU)\mathbf{1} \\ & - \l_ac_a - \l_bc_b \\
            \mathrm{s.t.} \ & \hat{\x} = -\frac{1}{2}Q^{-1}(q + \l_aA - \l_l + \l_u) \\
            & (\l_bB + \l_lL - \l_uU) \leq \mathbf{0} \\
            & \b = \mathbf{1}_{(b_\R)_i \in \mathrm{top}\{-(\l_bB + \l_lL - \l_uU)_i\}} \label{dualb}\\
            & \l_a, \l_b, \l_l, \l_u \in \R, \l_l, \l_u \geq \mathbf{0}.
    \end{align}
\end{theorem}
where (\ref{dualb}) means that we adopt the top $-(\l_bB + \l_lL - \l_uU)_i$ with index $i$ and set $b_i = 1$, and the rest set to $0$. A detailed proof is also shown below.

\begin{proof}
    For \begin{align*}
        \min_{\x, \b} \ & \x^TQ\x + (q + \l_aA - \l_l + \l_u)\x \\ & + (\l_bB + \l_lL - \l_uU) \b + (-\l_ac_a - \l_bc_b),
    \end{align*}
    we analyze seperately for $\x$ and $\b$:
    \begin{itemize}
        \item $\min_{\x, (\b)} \ \x^TQ\x + (q + \l_aA - \l_l + \l_u)\x$: If $Q$ is positive-definite, then let the gradient equals to $\mathbf{0}$, we have \begin{align*}
                2Q\hat{\x} + (q + \l_aA - \l_l + \l_u) = 0 \\ \implies \hat{\x} = -\frac{1}{2}Q^{-1}(q + \l_aA - \l_l + \l_u).
            \end{align*} Therefore, \begin{align*}
                \min_{\x, \b} \ & \x^TQ\x + (q + \l_aA - \l_l + \l_u)\x \\
                = \ & -\hat{\x}^TQ\hat{\x}
            \end{align*}
        \item $\min_{(\x), \b} \ (\l_bB + \l_lL - \l_uU) \b$: \begin{itemize}
            \item If $(\l_bB + \l_lL - \l_uU)_i \geq 0$, $b_i = 0$.
            \item If $(\l_bB + \l_lL - \l_uU)_i < 0$, $b_i = 1$.
            \end{itemize}
            Therefore, we could turn it to: \begin{align*}
                & (\l_bB + \l_lL - \l_uU)\mathbf{1} \\
                \mathrm{s.t.} \ & (\l_bB + \l_lL - \l_uU) \leq \mathbf{0}.
            \end{align*}
    \end{itemize}
    To sum up, the proof is complete.
\end{proof}

\begin{theorem}
    The upper bound for the Dual Model is simply the Prim Model, and we can derive the lower bound for by constructing a new matrix:
    \[\Phi_{i, j} = \begin{cases}
        \frac{1}{\sum_{k = 1}^{n}|Q^{-1}_{j, k}|} & i = j \\
        0 & i \neq j
    \end{cases}\]
    Then, the new model \begin{align*}
        \max_{\l}\min_{\x, \b} \ & \x^T\Phi\x + (q + \l_aA - \l_l + \l_u)\x \\ & + (\l_bB + \l_lL - \l_uU) \b + (-\l_ac_a - \l_bc_b)
    \end{align*} is the lower bound of the original Dual Model, i.e. \[\mathbf{0} < \Phi \leq Q.\]
\end{theorem}

Proof has been done by \citet{xu2024efficient}.

\iffalse
\begin{proof}
    Since $Q^{-1}$ is positive-definite because $Q$ is a positive-definite matrix, then by definition, \[\forall i \leq n(\sum_{j = 1}^{n}|Q^{-1}_{i, j}| > 0) \implies \forall j \leq n (\frac{1}{\sum_{k = 1}^{n}|Q^{-1}_{j, k}|} > 0).\] So $\Phi$ is also positive-definite. Therefore, we have \begin{align*}
        (Q^{-1}\Phi)_{i, j} & = \frac{Q_{i, j}^{-1}}{\sum_{k = 1}^{n}|Q^{-1}_{j, k}|} + 0 \\
        (Q^{-1}\Phi)_{i, j} & = \frac{Q_{i, j}^{-1}}{\sum_{k = 1}^{n}|Q^{-1}_{k, j}|} \quad \text{because of symmetry} \\
        & = 1 - \frac{\sum_{k = 1, k \neq j}^{n}|Q^{-1}_{i, k}|}{\sum_{k = 1}^{n}|Q^{-1}_{i, k}|} \\
        & \leq 1 \\
        \implies & Q^{-1}\Phi \leq I \\
        \implies & \Phi \leq Q
    \end{align*}
\end{proof}
\fi
\subsection{Augmented Model (Augm Model)}

\begin{definition}
    We define the initial augmented model by introducing augmented term upon our initial duality model as:
    \begin{align*}
        \max_{\l}\min_{\x, \b} \ & = \x^TQ\x + (q + \l_aA - \l_l + \l_u)\x \\ & + (\l_bB + \l_lL - \l_uU) \b - \l_ac_a - \l_bc_b \\
        & + \lambda_g |A\x - c_a|^2\\
        \mathrm{s.t.} \ & \l_a, \l_b, \l_l, \l_u \in \R, \l_l, \l_u \geq \mathbf{0}.
    \end{align*}
\end{definition}

\begin{theorem}
    Since it is similar to the Dual Model only with one more $\lambda_g |A\x - c_a|^2$ terms, thus we would like to make a further relaxation on this model. Deriving from Theorem 2.6, we use the diagonalized matrix $Phi$ to generate the final model as:
    \begin{align}
        \text{(Augm)}\max_{\l} \ & -\hat{\x}^T\Phi\hat{\x} + (\l_bB + \l_lL - \l_uU)\mathbf{1} \\ & - \l_ac_a - \l_bc_b + \lambda_g |A\x - c_a|^2 \\
        \mathrm{s.t.} \ & \Phi_{i, j} = \begin{cases}
            \frac{1}{\sum_{k = 1}^{n}|Q^{-1}_{j, k}|} & i = j \\
            0 & i \neq j
        \end{cases} \\
        & \hat{\x} = -\frac{1}{2}Q^{-1}(q + \l_aA - \l_l + \l_u) \\
        & (\l_bB + \l_lL - \l_uU) \leq \mathbf{0} \\
        & \l_a, \l_b, \l_l, \l_u \in \R, \l_l, \l_u \geq \mathbf{0}.
    \end{align}
\end{theorem}

\section{Proposed Solution Approach}
\label{appro}

\begin{figure*}[h!]
    \begin{center}
    %\framebox[4.0in]{$\;$}
    %\fbox{\rule[-.5cm]{0cm}{4cm} \rule[-.5cm]{4cm}{0cm}}
    \includegraphics[width=\linewidth]{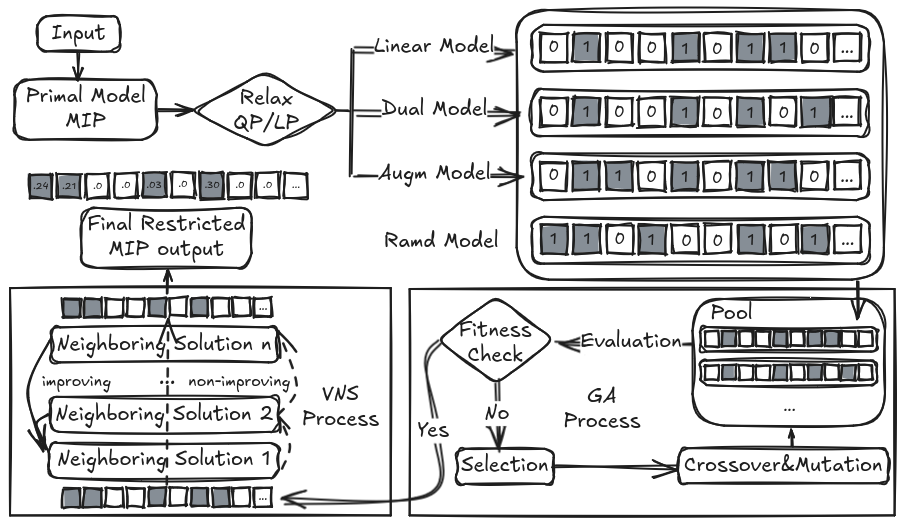}
    \end{center}
    \caption{A schematic visualizing our approach. Supported by Open-Source Excalidraw.}
    \label{fw}
\end{figure*}

To solve the problem, we employ a multistep approach integrating several methods as shown in Figure \ref{fw}. First, the Primal Model is relaxed into QP models. Starting with random initial solutions, with solutions generated by the QP models, a solution pool is created. Second, a Genetic Algorithm (GA) is applied to this pool, where a series of fitness evaluations are conducted, and the best GA solution is selected for further refinement. Third, the solution is improved using a neighborhood search algorithm, which enhances its accuracy. Finally, the refined solution from the neighborhood search is used to restrict the original MIQP to a QP formulation, yielding the optimal asset weights. The following sections provide a detailed explanation of each step.

\subsection{Initial Solution Pool}

For the input parameters, we require the expected returns for each asset and the covariance matrix between assets. Additionally, we must specify the target total return, the number of assets to select, and the upper and lower bounds for asset weights.

Given that Gurobi cannot efficiently handle large-scale MIQP problems, we first relax the Primal MIQP Model into three models: the Line Model, Dual Model, and Augm Model. Using the predefined parameters, we can generate initial solutions from these relaxed models. Simultaneously, we generate $M^\mathrm{Random}$ random initial solutions following these rules: the solution vector size should be $n$, each element in the vector must be either $0$ or $1$, exactly $k$ values in the vector should be $1$, and each value has an equal probability of being $1$. This gives us a total of $M = M^\mathrm{Relax} + M^\mathrm{Random}$ solutions, which are then assigned to the pool for further processing.

\subsection{Genetic Algorithm}

After generating solutions within the pool, we evaluate their competitiveness using a fitness function. This fitness function is defined as follows: given a solution vector $\b$, we introduce a new constraint into the Primal Model, requiring that the selected variables in the Primal Model match the values in $\b$. This transforms the problem into a QP. By solving this QP, we obtain the variance $v$, which is used as the measure of competitiveness. The fitness function is therefore $f(\b)=v$, where a lower value of $f(\b)$ indicates a more competitive solution.

Using this fitness function, we conduct a competition among the solutions in the pool. We retain the top $p\%$ of the solutions based on fitness and discard the rest. Next, we perform a fitness check, where the difference between the worst and best solutions must not be too large. Specifically, we require that \[l \geq \frac{\max\{f(\b) | \b \in \mathrm{pool}\} - \min\{f(\b) | \b \in \mathrm{pool}\}}{\min\{f(\b) | \b \in \mathrm{pool}\}}\] for solutions within the pool. If this condition is met, the best solution is selected for the next step. Otherwise, two-parent solutions are randomly selected, and we apply a crossover operation to generate a new solution.

The crossover is performed under the following rules: the new solution vector size is $n$, each element in the vector must be either $0$ or $1$, and exactly $k$ elements should be $1$. If both parents have $b_i^\mathrm{father} = b_i^\mathrm{mother} = 1$, then $b_i^\mathrm{son} = 1$. For positions where $b_i^\mathrm{father} \neq b_i^\mathrm{mother}$, $b_i^\mathrm{son}$ is assigned a value of $1$ with equal probability. The new solution is also subjected to a mutation with a probability of $m\%$, where one index with $b_i = 1$ and another with $b_i = 0$ are swapped.

Once the new solution is generated, it, along with its parent solutions, is reintroduced into the pool, and the competition is repeated. This process continues iteratively until the fitness check condition is satisfied.

\subsection{Neighborhood Search Algorithm}

With the best solution generated by the Genetic Algorithm (GA) approach, which is either very close to or equal to the optimal primal solution, we can apply neighborhood searching to refine the solution further. In this approach, the GA solution is initially treated as the ``neighborhood solution'' or neighboring solution 1.

To generate neighboring solution 2, we randomly select two values in the neighborhood solution: one with a value of $1$ and another with a value of $0$. We then swap their positions. After that, we compare the fitness of neighboring solution 2 with that of the initial neighborhood solution. If the fitness of neighboring solution 2 is better (i.e., lower), it becomes the new neighborhood solution, and we repeat the process.

If neighboring solution 2 is not better, we proceed by generating neighboring solution 3 in the same manner—swapping the values of $1$ and $0$ based on neighboring solution 2 and comparing it with the neighborhood solution. This iterative process continues until either an improvement is found or a predefined limit is reached.

\section{Experiment Results}
\label{resul}

\subsection{Dataset Description}

Our approach and later gap analysis are based on weekly price data from March 1992 to September 1997 of the Hang Seng (Hong Kong), DAX 100 (Germany), FTSE 100 (UK), S\&P 100 (USA) and Nikkei 225 (Japan) \citep{chang2000heuristics}. The size of each dataset is ranged from $n = 31$ for Hang Seng to $n = 225$ for Nikkei. We set the global upper bound for each asset as $u_i = 1$ and the lower bound as $l_i = 0.01$, the total number of assets we would choose is $k = 10$, and $\lambda_g$ in Augm Model we adopt $1e-7$. More detailed information is available in the code will be available after acceptance.

\subsection{Environment Description}

The experiment is delivered under 96 Intel(R) Xeon(R) CPU @ 2.00GHz. The Python version is 3.12.8 and the Gurobi version is 12.0.0 build v12.0.0rc1 (linux64) with an Academic license.

\subsection{Efficient Frontiers and Percentage Errors}
\iffalse
\begin{figure}[h!]
    \begin{center}
    %\framebox[4.0in]{$\;$}
    %\fbox{\rule[-.5cm]{0cm}{4cm} \rule[-.5cm]{4cm}{0cm}}
    \includegraphics[width=\linewidth]{ef-3.png}
    \end{center}
    \caption{A sample (port 3) of EF and solutions generated by different approaches. Supported by Matplotlib.}
\end{figure}
\fi
Based on 5 datasets above, the given unconstrained efficient frontier (UEF) as an upper bound is drawn as a line in each figure. For comparing and visualizing the effect of Line, Dual, Augm, and Ours Models, we sample 50 points respecting returns from the domain of the UEF and then mark the points on each figure. In all the figures, the risks have a positive correlation with returns, which means higher returns come with higher risks, and our approach based on Line, Dual, and Augm Models should remain left which is the closest to the upper bound line. The features shown in the figures logically make sense.

\begin{table}[h!]
    \caption{Percentage Errors of different approaches}
    \label{table-PE}
    \vskip 0.15in
    \begin{center}
    \begin{tabular}{llllll}
    \toprule
    PE && Line & Dual & Augm & Ours \\
    \midrule
    1 
    & Mean   & 0.7929 & 0.5232 & 0.9131 & \bf 0.5231 \\
    & Median & 0.4317 & 0.4147 & 0.6223 & \bf 0.3356 \\
    \\
    2 
    & Mean   & 4.8232 & 5.8435 & 10.3267 & \bf 0.6050 \\
    & Median & 4.3321 & 2.9218 & 6.0821 & \bf 0.1502 \\
    \\
    3 
    & Mean   & 4.7993 & 4.0280 & 6.2230 & \bf 0.1598 \\
    & Median & 5.3894 & 2.8866 & 9.7862 & \bf 0.3217 \\
    \\
    4 
    & Mean   & 5.5428 & 5.8134 & 8.9946 & \bf 0.2272 \\
    & Median & 4.2753 & 5.3949 & 11.0371 & \bf 0.9978 \\
    \\
    5
    & Mean   & 1.3079 & 1.7143 & 2.0529 & \bf 0.1368 \\
    & Median & 0.4247 & 0.3827 & 1.7317 & \bf 0.1363 \\
    \bottomrule
    \end{tabular}
    \end{center}
\end{table}

Then Table \ref{table-PE} reflects the percentage error from point to the EF line in the figures. In the table, Line, Dual, and Augm models as relaxation models cannot always have a good effect on solution generation. However, our approach based on these exact solutions and with heuristics gains a good effect. Compared with Gurobi optimal solutions and other heuristics \citep{woodside-oriakhi2011heuristic}, our approach gains a close but mostly a better effect. To be more specific, we mostly reach a mean percentage error, implying that our method has fewer outliners and thus has more stable results.

It might be noticed that although we apply GA in our methods, we cannot perform a better result than GA. It is because of the limitations of our current devices. To improve the speed and reveal more obvious solutions in the figures above, we make some simplifications with GA and design a different structure compared with the state-of-the-art GA. It is the same as our neighborhood searching algorithm. Better results might be approached after applying more recent heuristics with higher computational devices.

\subsection{Models Gap Formulation}

\subsubsection{Binary Gap}

\def\g{\mathbf{g}}

Since we always compare and analyze our gaps between one of the relaxed models and the Gurobi optimal Model, we define the gaps for them as $\g^\mathrm{Relaxed} = \b^\mathrm{Relaxed} - \b^\mathrm{Gurobi}$, where $\b^\mathrm{Relaxed}$ stands for the choice binary variables' solution for one of the relaxed models, and $\b^\mathrm{Gurobi}$ stands for Gurobi optimal's. Since $\mathbf{d}^\mathrm{Relaxed}$ is a differences vector between two binary vectors, we define variable $g^\mathrm{Relaxed} = \frac{1}{2}|\g^\mathrm{Relaxed}| = \frac{1}{2}\sum_{i = 1}^{n} |b^\mathrm{Relaxed}_i - b^\mathrm{Gurobi}_i|$ to analyze the number of all differences between two solutions. To be more specific:
\begin{equation}
    \begin{aligned}
        \g^\mathrm{Line} & = \b^\mathrm{Line} - \b^\mathrm{Gurobi} \\
        \g^\mathrm{Dual} & = \b^\mathrm{Dual} - \b^\mathrm{Gurobi} \\
        \g^\mathrm{Augm} & = \b^\mathrm{Augm} - \b^\mathrm{Gurobi} \\
        \g^\mathrm{Ours} & = \b^\mathrm{Augm} - \b^\mathrm{Gurobi}
    \end{aligned}
\end{equation}
In this way, we are able to quantify and analyze the binary gaps. Our results are shown in Table \ref{table-bgaps}.

% others, bold best

\begin{table}[h!]
    \caption{Binary Gaps of different approaches}
    \label{table-bgaps}
    \vskip 0.15in
    \begin{center}
    \begin{tabular}{llllll}
    \toprule
    $\g$ && Line & Dual & Augm & Ours \\
    \midrule
    1 
    & Mean   & 0.5531 & 0.5957 & 1.7447 & \bf 0.0213 \\
    & Median & 0.0000 & 0.0000 & 2.0000 & \bf 0.0000 \\
    & Max    & \bf 2.0000 & 4.0000 & 8.0000 & \bf 2.0000 \\
    & Min    & \bf 0.0000 & \bf 0.0000 & \bf 0.0000 & \bf 0.0000 \\
    \\
    2
    & Mean   & 3.0833 & 3.3750 & 7.4166 & \bf 2.4583 \\
    & Median & \bf 2.0000 & \bf 2.0000 & 8.0000 & \bf 2.0000 \\
    & Max    & \bf 10.000 & 14.000 & \bf 10.000 & 14.000 \\
    & Min    & \bf 0.0000 & \bf 0.0000 & 2.0000 & \bf 0.0000 \\
    \\
    3
    & Mean   & 3.0416 & 4.5833 & 6.6666 & \bf 2.8333 \\
    & Median & \bf 2.0000 & 4.0000 & 6.0000 & \bf 2.0000 \\
    & Max    & \bf 10.000 & 14.000 & 12.000 & \bf 10.000 \\
    & Min    & \bf 0.0000 & \bf 0.0000 & 2.0000 & \bf 0.0000 \\
    \\
    4
    & Mean   & 3.9591 & 5.5918 & 6.0816 & \bf 3.5918 \\
    & Median & \bf 2.0000 & 4.0000 & 6.0000 & \bf 2.0000 \\
    & Max    & 12.0000 & 20.0000 & 12.0000 & \bf 10.0000 \\
    & Min    & \bf 0.0000 & \bf 0.0000 & 2.0000 & \bf 0.0000 \\
    \\
    5
    & Mean   & 18.053 & 20.605 & 27.249 & \bf 6.4263 \\
    & Median & 10.000 & 14.000 & 20.000 & \bf 2.0000 \\
    & Max    & 36.000 & 52.000 & 78.000 & \bf 30.0000 \\
    & Min    & \bf 0.0000 & \bf 0.0000 & 2.0000 & \bf 0.0000 \\
    \bottomrule
    \end{tabular}
    \end{center}
\end{table}

Table \ref{table-bgaps} reflects the relations between the solution generated by relaxed models and our models with the Gurobi Optimal solution. Most of them reflect that our approach is closer to Gurobi optimal than Line, Dual, and Augm Models.

\iffalse

\begin{table}[h!]
    \caption{The sample point of gaps in port 1}
    \label{partial-solu}
    \begin{center}
    \begin{tabular}{l|ll}
    Method & Solution & Risk-Variance \\
    \hline \\
    Linear & [0, 0, 0, ..., 0, 1, 0, ..., 1, 1, 1, 0, ..., 0, 0, 1, 1, 1] & 0.000759962 \\
    Dual   & [0, 0, 0, ..., 0, 1, 0, ..., 1, 1, 0, 0, ..., 0, 1, 1, 1, 1] & 0.000650130 \\
    Augm   & [0, 1, 0, ..., 0, 1, 0, ..., 1, 0, 0, 0, ..., 0, 1, 1, 1, 1] & 0.000657009 \\
    Ours   & [0, 0, 0, ..., 0, 1, 0, ..., 1, 1, 0, 0, ..., 0, 1, 1, 1, 1] & \bf 0.000650130 \\
    \end{tabular}
    \end{center}
\end{table}

\fi

\subsubsection{Objective Gap}

However, it is not enough to just analyze the binary gaps considering some solutions with different binaries might reach even better objective values. Therefore, we introduced the objective gap:
\begin{equation}
    \begin{aligned}
        g^\mathrm{Line} & = \frac{obj^\mathrm{Line} - obj^\mathrm{Gurobi}}{obj^\mathrm{Gurobi}} \\
        g^\mathrm{Dual} & = \frac{obj^\mathrm{Dual} - obj^\mathrm{Gurobi}}{obj^\mathrm{Gurobi}} \\
        g^\mathrm{Augm} & = \frac{obj^\mathrm{Augm} - obj^\mathrm{Gurobi}}{obj^\mathrm{Gurobi}} \\
        g^\mathrm{Ours} & = \frac{obj^\mathrm{Ours} - obj^\mathrm{Gurobi}}{obj^\mathrm{Gurobi}} \\
    \end{aligned}
\end{equation}
where $g^{A}$ stands for the objective gaps between Model A and the optimal primal values of Gurobi, $obj$ stands for the objective value, meaning with the binaries solution in each model, adopting them as an upper bound and lower bound, we would eliminate the cardinality constraints and gain a programming without integer. The objective value $obj$ is coming from the programming.

\begin{figure}[h!]
    \label{port1}
    \begin{center}
    %\framebox[4.0in]{$\;$}
    %\fbox{\rule[-.5cm]{0cm}{4cm} \rule[-.5cm]{4cm}{0cm}}
    \includegraphics[width=\linewidth]{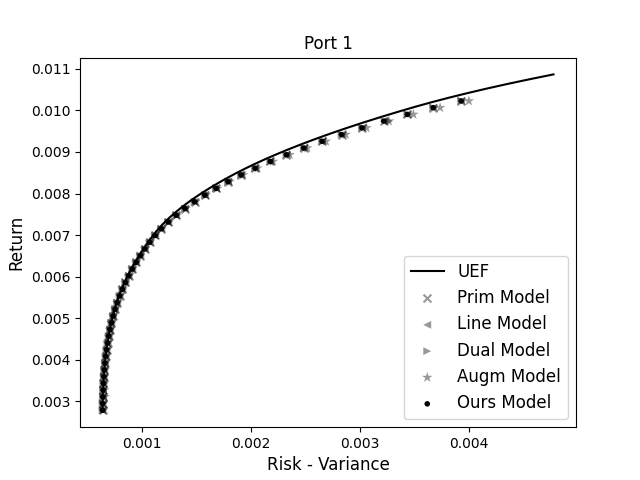}
    \end{center}
    \caption{A schematic visualizing the objective gaps. Supported by Matplotlib.}
\end{figure}

\begin{table}[h!]
    \caption{Objective Gaps of different approaches}
    \label{table-ogaps}
    \vskip 0.15in
    \begin{center}
    \begin{tabular}{llllll}
    \toprule
    g & (\%) & Line & Dual & Augm & Ours \\
    \midrule
    1 
    & Mean   & -0.0134 & 0.0320 & 0.6199 & \bf -0.0139 \\
    & Median & \bf 0.0000 & \bf 0.0000 & 0.4428 & \bf 0.0000 \\
    & Max    & 0.0196 & 0.8544 & 2.8169 & \bf 0.0000 \\
    & Min & \bf -0.1121 & \bf -0.1121 & \bf -0.1121 & \bf -0.1121 \\
    \\
    2
    & Mean   & 3.7400 & 6.4104 & 17.085 & \bf 0.4146 \\
    & Median & 0.5951 & 1.1207 & 10.876 & \bf 0.1993 \\
    & Max    & 25.665 & 68.956 & 89.733 & \bf 8.2481 \\
    & Min    & -0.2863 & -0.2863 & 0.1744 & \bf -28.742 \\
    \\
    3
    & Mean   & 1.8865 & 7.1643 & 7.6136 & \bf 0.7935 \\
    & Median & 0.2503 & 1.1735 & 3.8920 & \bf 0.2287 \\
    & Max    & 20.266 & 42.294 & 34.663 & \bf 7.9213 \\
    & Min    & \bf -0.1837 & \bf -0.1837 & 0.3468 & \bf -0.1837 \\
    \\
    4
    & Mean   & 5.3076 & 11.036 & 11.022 & \bf 1.8449 \\
    & Median & 1.5950 & 5.8239 & 6.5265 & \bf 0.8271 \\
    & Max    & 31.312 & 72.097 & 55.262 & \bf 7.4657 \\
    & Min    & \bf -0.0671 & -0.0585 & 0.0696 & \bf -0.0671 \\
    \\
    5
    & Mean   & 1.4361 & 13.923 & 24.931 & \bf 0.2133 \\
    & Median & 0.7066 & 2.1052 & 13.253 & \bf 0.2098 \\
    & Max    & 18.773 & 53.256 & 60.623 & \bf 10.478 \\
    & Min    & -0.9378 & -0.4020 & 0.3124 & \bf -4.2573 \\
    \bottomrule
    \end{tabular}
    \end{center}
\end{table}

Table \ref{table-ogaps} reflects the relations between objective values generated by relaxed models and our models with the Gurobi Optimal solution. Most of them reflect that our approach is closer to Gurobi optimal than Line, Dual, and Augm Models.

To be more specific, in most cases, our model would reach the best objective values, that is the lowest variance among other relaxation methods. More importantly, the negative objective gaps such as the data in Port 1 reveal our flamework can outperform the Gurobi and approach a better solution, implying a further application for MIP programming usage.

\section{Conclusion}
\label{concl}

In this paper, focusing on the binary variables of MIQP, we discuss the properties and features of exact relaxation models, the Line Model, the Dual Model, and the Augm Model. Then, we design an approach based on the solutions generated by exact relaxation models with heuristic methods, a Genetic Algorithm, and a Neighborhood Search Algorithm. The exact techniques can approach the exact solutions, then with the heuristics strengthen methods, it enables the structure to find the solution around the given solutions. With our approach, we test it with the partial exact relaxation models on the datasets of \citep{chang2000heuristics}, comparing with commercial solver Gurobi and others heuristics by \citet{woodside-oriakhi2011heuristic}, and get a state-of-the-art results. The results show that our approach is effective and can better find the solution for MIQP, hence our approach offers a way for finding solutions by combining exact relaxations and heuristics. Because the pages are limited, we are only able to discuss one application in detail. Still, this method could be able to be generalized into MIP as an overall structure for solving MIP problems. That is, using multiple relaxation methods to generate an initial solution pool and then using the pool with heuristic methods to approach state-of-the-art results. 

\bibliography{ref}

@article{markowitz1952portfolio,
  title = {Portfolio {{Selection}}},
  author = {Markowitz, Harry},
  year = {1952},
  journal = {The Journal of Finance},
  volume = {7},
  number = {1},
  pages = {77--91},
  issn = {1540-6261},
  doi = {10.1111/j.1540-6261.1952.tb01525.x},
  urldate = {2024-08-29},
  copyright = {{\copyright} 1952 the American Finance Association}
}

@article{chang2000heuristics,
  title = {Heuristics for Cardinality Constrained Portfolio Optimisation},
  author = {Chang, T. -J. and Meade, N. and Beasley, J. E. and Sharaiha, Y. M.},
  year = {2000},
  journal = {Computers \& Operations Research},
  volume = {27},
  number = {13},
  pages = {1271--1302},
  issn = {0305-0548},
  doi = {10.1016/S0305-0548(99)00074-X},
  urldate = {2024-06-05}
}

@article{pia2017mixedinteger,
  title = {Mixed-Integer Quadratic Programming Is in {{NP}}},
  author = {Pia, Alberto Del and Dey, Santanu S. and Molinaro, Marco},
  year = {2017},
  journal = {Mathematical Programming},
  volume = {162},
  number = {1},
  pages = {225--240},
  issn = {1436-4646},
  doi = {10.1007/s10107-016-1036-0},
  urldate = {2024-08-29}
}

@article{woodside-oriakhi2011heuristic,
  title = {Heuristic Algorithms for the Cardinality Constrained Efficient Frontier},
  author = {{Woodside-Oriakhi}, M. and Lucas, C. and Beasley, J. E.},
  year = {2011},
  journal = {European Journal of Operational Research},
  volume = {213},
  number = {3},
  pages = {538--550},
  issn = {0377-2217},
  doi = {10.1016/j.ejor.2011.03.030},
  urldate = {2024-05-30}
}

@article{xu2024efficient,
  title = {An {{Efficient Global Optimal Method}} for {{Cardinality Constrained Portfolio Optimization}}},
  author = {Xu, Wei and Tang, Jie and Yiu, Ka Fai Cedric and Peng, Jian Wen},
  year = {2024},
  journal = {INFORMS Journal on Computing},
  volume = {36},
  number = {2},
  pages = {690--704},
  publisher = {INFORMS},
  issn = {1091-9856},
  doi = {10.1287/ijoc.2022.0344},
  urldate = {2024-06-29}
}

@article{markowitz1956optimization,
  title = {The Optimization of a Quadratic Function Subject to Linear Constraints},
  author = {Markowitz, Harry},
  year = {1956},
  journal = {Naval Research Logistics Quarterly},
  volume = {3},
  number = {1-2},
  pages = {111--133},
  issn = {1931-9193},
  doi = {10.1002/nav.3800030110},
  urldate = {2024-09-09},
  copyright = {Copyright {\copyright} 1956 Wiley Periodicals, Inc., A Wiley Company}
}

@article{graziasperanza1996heuristic,
  title = {A Heuristic Algorithm for a Portfolio Optimization Model Applied to the {{Milan}} Stock Market},
  author = {Grazia Speranza, M.},
  year = {1996},
  journal = {Computers \& Operations Research},
  volume = {23},
  number = {5},
  pages = {433--441},
  issn = {0305-0548},
  doi = {10.1016/0305-0548(95)00030-5},
  urldate = {2024-09-10}
}

@article{deng2012markowitzbased,
  title = {Markowitz-Based Portfolio Selection with Cardinality Constraints Using Improved Particle Swarm Optimization},
  author = {Deng, Guang-Feng and Lin, Woo-Tsong and Lo, Chih-Chung},
  year = {2012},
  journal = {Expert Systems with Applications},
  volume = {39},
  number = {4},
  pages = {4558--4566},
  issn = {0957-4174},
  doi = {10.1016/j.eswa.2011.09.129},
  urldate = {2024-08-29}
}

@article{lwin2013hybrid,
  title = {A Hybrid Algorithm for Constrained Portfolio Selection Problems},
  author = {Lwin, Khin and Qu, Rong},
  year = {2013},
  journal = {Applied Intelligence},
  volume = {39},
  number = {2},
  pages = {251--266},
  issn = {1573-7497},
  doi = {10.1007/s10489-012-0411-7},
  urldate = {2024-08-29}
}

@article{tuba2014artificial,
  title = {Artificial {{Bee Colony Algorithm Hybridized}} with {{Firefly Algorithm}} for {{Cardinality Constrained Mean-Variance Portfolio Selection Problem}}},
  author = {Tuba, Milan and Bacanin, Nebojsa},
  year = {2014},
  journal = {Applied Mathematics \& Information Sciences},
  volume = {8},
  number = {6},
  pages = {2831--2844},
  issn = {1935-0090, 2325-0399},
  doi = {10.12785/amis/080619},
  urldate = {2024-08-29}
}

@article{kalayci2017artificial,
  title = {An Artificial Bee Colony Algorithm with Feasibility Enforcement and Infeasibility Toleration Procedures for Cardinality Constrained Portfolio Optimization},
  author = {Kalayci, Can B. and Ertenlice, Okkes and Akyer, Hasan and Aygoren, Hakan},
  year = {2017},
  journal = {Expert Systems with Applications},
  volume = {85},
  pages = {61--75},
  issn = {0957-4174},
  doi = {10.1016/j.eswa.2017.05.018},
  urldate = {2024-08-29}
}

@article{akbay2020parallel,
  title = {A Parallel Variable Neighborhood Search Algorithm with Quadratic Programming for Cardinality Constrained Portfolio Optimization},
  author = {Akbay, Mehmet Anil and Kalayci, Can B. and Polat, Olcay},
  year = {2020},
  journal = {Knowledge-Based Systems},
  volume = {198},
  pages = {105944},
  issn = {0950-7051},
  doi = {10.1016/j.knosys.2020.105944},
  urldate = {2024-08-29}
}

@article{kalayci2020efficient,
  title = {An Efficient Hybrid Metaheuristic Algorithm for Cardinality Constrained Portfolio Optimization},
  author = {Kalayci, Can B. and Polat, Olcay and Akbay, Mehmet A.},
  year = {2020},
  journal = {Swarm and Evolutionary Computation},
  volume = {54},
  pages = {100662},
  issn = {2210-6502},
  doi = {10.1016/j.swevo.2020.100662},
  urldate = {2024-08-29}
}

@article{shaw2008lagrangian,
  title = {Lagrangian Relaxation Procedure for Cardinality-Constrained Portfolio Optimization},
  author = {Shaw, Dong X. and Liu, Shucheng and Kopman, Leonid},
  year = {2008},
  journal = {Optimization Methods and Software},
  volume = {23},
  number = {3},
  pages = {411--420},
  publisher = {Taylor \& Francis},
  issn = {1055-6788},
  doi = {10.1080/10556780701722542},
  urldate = {2024-06-05}
}

@article{guzelsoy2007duality,
  title = {Duality for {{Mixed-Integer Linear Programs}}},
  author = {Guzelsoy, Menal and Ralphs, Theodore K},
  year = {2007},
  journal = {International Journal of Operations Research},
  volume = {4},
  number = {3},
  pages = {118--137}
}

@incollection{guzelsoy2011integer,
  title = {Integer {{Programming Duality}}},
  booktitle = {Wiley {{Encyclopedia}} of {{Operations Research}} and {{Management Science}}},
  author = {Guzelsoy, Menal and Ralphs, Ted K.},
  year = {2011},
  publisher = {John Wiley \& Sons, Ltd},
  doi = {10.1002/9780470400531.eorms0413},
  urldate = {2024-08-28},
  copyright = {Copyright {\copyright} 2010 John Wiley \& Sons, Inc. All rights reserved.},
  isbn = {978-0-470-40053-1}
}

@article{feizollahi2017exact,
  title = {Exact Augmented {{Lagrangian}} Duality for Mixed Integer Linear Programming},
  author = {Feizollahi, Mohammad Javad and Ahmed, Shabbir and Sun, Andy},
  year = {2017},
  journal = {Mathematical Programming},
  volume = {161},
  number = {1},
  pages = {365--387},
  issn = {1436-4646},
  doi = {10.1007/s10107-016-1012-8},
  urldate = {2024-08-28}
}

@article{baykasoglu2015grasp,
  title = {A {{GRASP}} Based Solution Approach to Solve Cardinality Constrained Portfolio Optimization Problems},
  author = {Baykaso{\u g}lu, Adil and Yunusoglu, Mualla Gonca and Burcin {\"O}zsoydan, F.},
  year = {2015},
  journal = {Computers \& Industrial Engineering},
  volume = {90},
  pages = {339--351},
  issn = {0360-8352},
  doi = {10.1016/j.cie.2015.10.009},
  urldate = {2024-08-29}
}

@article{ertenlice2018survey,
  title = {A Survey of Swarm Intelligence for Portfolio Optimization: {{Algorithms}} and Applications},
  shorttitle = {A Survey of Swarm Intelligence for Portfolio Optimization},
  author = {Ertenlice, Okkes and Kalayci, Can B.},
  year = {2018},
  journal = {Swarm and Evolutionary Computation},
  volume = {39},
  pages = {36--52},
  issn = {2210-6502},
  doi = {10.1016/j.swevo.2018.01.009},
  urldate = {2024-09-13}
}

@article{quirynen2024realtime,
  title = {Real-{{Time Mixed-Integer Quadratic Programming}} for {{Vehicle Decision-Making}} and {{Motion Planning}}},
  author = {Quirynen, Rien and Safaoui, Sleiman and Cairano, Stefano Di},
  year = {2024},
  journal = {IEEE Transactions on Control Systems Technology},
  pages = {1--0},
  issn = {1558-0865},
  doi = {10.1109/TCST.2024.3449703},
  urldate = {2024-09-19}
}

@article{ajagekar2022hybrid,
  title = {Hybrid {{Classical-Quantum Optimization Techniques}} for {{Solving Mixed-Integer Programming Problems}} in {{Production Scheduling}}},
  author = {Ajagekar, Akshay and Hamoud, Kumail Al and You, Fengqi},
  year = {2022},
  journal = {IEEE Transactions on Quantum Engineering},
  volume = {3},
  pages = {1--16},
  issn = {2689-1808},
  doi = {10.1109/TQE.2022.3187367},
  urldate = {2024-09-19}
}

@article{anderson2020strong,
  title = {Strong Mixed-Integer Programming Formulations for Trained Neural Networks},
  author = {Anderson, Ross and Huchette, Joey and Ma, Will and Tjandraatmadja, Christian and Vielma, Juan Pablo},
  year = {2020},
  journal = {Mathematical Programming},
  volume = {183},
  number = {1},
  pages = {3--39},
  issn = {1436-4646},
  doi = {10.1007/s10107-020-01474-5},
  urldate = {2024-09-19}
}

@article{ferber2020mipaal,
  title = {{{MIPaaL}}: {{Mixed Integer Program}} as a {{Layer}}},
  shorttitle = {{{MIPaaL}}},
  author = {Ferber, Aaron and Wilder, Bryan and Dilkina, Bistra and Tambe, Milind},
  year = {2020},
  journal = {Proceedings of the AAAI Conference on Artificial Intelligence},
  volume = {34},
  number = {02},
  pages = {1504--1511},
  issn = {2374-3468},
  doi = {10.1609/aaai.v34i02.5509},
  urldate = {2024-09-19},
  copyright = {Copyright (c) 2020 Association for the Advancement of Artificial Intelligence}
}

@inproceedings{wang2022emotion,
  title = {Emotion {{Recognition Based}} on~{{Multi-scale Convolutional Neural Network}}},
  booktitle = {Data {{Mining}} and {{Big Data}}},
  author = {Wang, Zeen},
  editor = {Tan, Ying and Shi, Yuhui},
  year = {2022},
  pages = {152--164},
  publisher = {Springer Nature},
  address = {Singapore},
  doi = {10.1007/978-981-19-9297-1_12},
  copyright = {All rights reserved},
  isbn = {978-981-19929-7-1}
}

@article{wang2024enhancing,
  title = {Enhancing Ocular Diseases Recognition with Domain Adaptive Framework: Leveraging Domain Confusion},
  shorttitle = {Enhancing Ocular Diseases Recognition with Domain Adaptive Framework},
  author = {Wang, Zayn},
  year = {2024},
  journal = {International Journal of Machine Learning and Cybernetics},
  pages = {1--9},
  issn = {1868-808X},
  doi = {10.1007/s13042-024-02358-2},
  urldate = {2024-08-30}
}

@article{fisher1981lagrangian,
  title = {The {{Lagrangian Relaxation Method}} for {{Solving Integer Programming Problems}}},
  author = {Fisher, Marshall L.},
  year = {1981},
  journal = {Management Science},
  volume = {27},
  number = {1},
  pages = {1--18},
  publisher = {INFORMS},
  issn = {0025-1909},
  doi = {10.1287/mnsc.27.1.1},
  urldate = {2024-09-19}
}

@incollection{nemhauser1988duality,
  title = {Duality and {{Relaxation}}},
  booktitle = {Integer and {{Combinatorial Optimization}}},
  author = {Nemhauser, George and Wolsey, Laurence},
  year = {1988},
  pages = {296--348},
  publisher = {John Wiley \& Sons, Ltd},
  doi = {10.1002/9781118627372.ch10},
  urldate = {2024-08-28},
  chapter = {II.3},
  copyright = {Copyright {\copyright} 1999 John Wiley \& Sons, Inc.},
  isbn = {978-1-118-62737-2}
}

@article{li2006optimal,
  title = {Optimal {{Lot Solution}} to {{Cardinality Constrained Mean}}--{{Variance Formulation}} for {{Portfolio Selection}}},
  author = {Li, Duan and Sun, Xiaoling and Wang, Jun},
  year = {2006},
  journal = {Mathematical Finance},
  volume = {16},
  number = {1},
  pages = {83--101},
  issn = {1467-9965},
  doi = {10.1111/j.1467-9965.2006.00262.x},
  urldate = {2024-09-19}
}

@article{bragin2024survey,
  title = {Survey on {{Lagrangian}} Relaxation for {{MILP}}: Importance, Challenges, Historical Review, Recent Advancements, and Opportunities},
  shorttitle = {Survey on {{Lagrangian}} Relaxation for {{MILP}}},
  author = {Bragin, Mikhail A.},
  year = {2024},
  journal = {Annals of Operations Research},
  volume = {333},
  number = {1},
  pages = {29--45},
  issn = {1572-9338},
  doi = {10.1007/s10479-023-05499-9},
  urldate = {2024-12-02}
}
\bibliographystyle{icml2025}

\iffalse
%%%%%%%%%%%%%%%%%%%%%%%%%%%%%%%%%%%%%%%%%%%%%%%%%%%%%%%%%%%%%%%%%%%%%%%%%%%%%%%
%%%%%%%%%%%%%%%%%%%%%%%%%%%%%%%%%%%%%%%%%%%%%%%%%%%%%%%%%%%%%%%%%%%%%%%%%%%%%%%
% APPENDIX
%%%%%%%%%%%%%%%%%%%%%%%%%%%%%%%%%%%%%%%%%%%%%%%%%%%%%%%%%%%%%%%%%%%%%%%%%%%%%%%
%%%%%%%%%%%%%%%%%%%%%%%%%%%%%%%%%%%%%%%%%%%%%%%%%%%%%%%%%%%%%%%%%%%%%%%%%%%%%%%
\newpage
\appendix
\onecolumn
\section{You \emph{can} have an appendix here.}

You can have as much text here as you want. The main body must be at most $8$ pages long.
For the final version, one more page can be added.
If you want, you can use an appendix like this one.  

The $\mathtt{\backslash onecolumn}$ command above can be kept in place if you prefer a one-column appendix, or can be removed if you prefer a two-column appendix.  Apart from this possible change, the style (font size, spacing, margins, page numbering, etc.) should be kept the same as the main body.
%%%%%%%%%%%%%%%%%%%%%%%%%%%%%%%%%%%%%%%%%%%%%%%%%%%%%%%%%%%%%%%%%%%%%%%%%%%%%%%
%%%%%%%%%%%%%%%%%%%%%%%%%%%%%%%%%%%%%%%%%%%%%%%%%%%%%%%%%%%%%%%%%%%%%%%%%%%%%%%
\fi

\end{document}